\documentclass[hidelinks,11pt]{article}
\usepackage{amsmath,amssymb,amsfonts,amsthm,mathrsfs,relsize,mathtools,graphicx,float,hyperref,yhmath,stmaryrd,url,multicol,authblk,fullpage,bbm,euscript}
\usepackage[labelsep=period,font=footnotesize,labelfont=bf]{caption}
\usepackage[title]{appendix}
\usepackage{titlesec}
\titlelabel{\thetitle.\,\,}
\allowdisplaybreaks
\usepackage{pgf,tikz}
\usetikzlibrary{arrows}
\usetikzlibrary[patterns]

\let\svthefootnote\thefootnote
\newcommand\blankfootnote[1]{%
\let\thefootnote\relax\footnotetext{#1}%
\let\thefootnote\svthefootnote%
}

\theoremstyle{plain}
\newtheorem{theorem}{Theorem}[section]
\newtheorem{lemma}[theorem]{Lemma}
\newtheorem{claim}[theorem]{Claim}

\newtheorem{corollary}[theorem]{Corollary}

\newtheorem{observation}[theorem]{Observation}

\theoremstyle{definition}

\newtheorem{remark}[theorem]{Remark}

\newcommand{\da}{\rotatebox[origin=c]{45}{$\Box$}}

\DeclareMathAlphabet{\mathbbmsl}{U}{bbm}{m}{sl}

\DeclareMathAlphabet{\mathpzc}{OT1}{pzc}{m}{it}
\DeclareMathAlphabet{\mathsfit}{T1}{\sfdefault}{\mddefault}{\sldefault}\SetMathAlphabet{\mathsfit}{bold}{T1}{\sfdefault}{\bfdefault}{\sldefault}

\makeatletter
\newcommand\appendix@section[1]{%
\refstepcounter{section}%
\orig@section*{Appendix \@Alph\c@section. #1}%
\addcontentsline{toc}{section}{Appendix \@Alph\c@section. #1}%
}
\let\orig@section\section
\g@addto@macro\appendix{\let\section\appendix@section}
\makeatother

\begin{document}

\title{\hspace{-1.5mm}{\bf Saturation numbers of   bipartite graphs in random graphs}\\ \vspace{1cm}}

\author{
\hspace{-0.5mm}Meysam Miralaei$^{^{1, a}}$ \,   Ali Mohammadian$^{^{2,b,c}}$ \,  Behruz Tayfeh-Rezaie$^{^{1,b}}$ \,     Maksim Zhukovskii$^{^3}$ \\ \vspace{1mm}
$^{^1}$School of Mathematics,    Institute for Research in
Fundamental  Sciences (IPM),  \\  P.O. Box 19395-5746, Tehran, Iran \\ \vspace{1mm}
$^{^2}$School of Mathematical Sciences,    Anhui University,  \\
Hefei 230601,  Anhui,    China \\ \vspace{1mm}
$^{^3}$Department of Computer Science,   University of Sheffield,     \\
Sheffield S1 4DP, UK    \\ \vspace{2mm}
\href{mailto:m.miralaei@ipm.ir}{m.miralaei@ipm.ir} \quad  \href{mailto:ali\_m@ahu.edu.cn}{ali\_m@ahu.edu.cn} \quad     \href{mailto:tayfeh-r@ipm.ir}{tayfeh-r@ipm.ir} \quad  \href{mailto:m.zhukovskii@sheffield.ac.uk}{m.zhukovskii@sheffield.ac.uk} \\ \vspace{1cm}}

\blankfootnote{\hspace{-6mm}$^{^a}$Partially  supported by a grant from IPM.\\
$^{^b}$Partially  supported by Iran  National  Science Foundation   under project number  99003814.\\
$^{^c}$Partially  supported   by the Natural Science Foundation of Anhui Province  with  grant identifier 2008085MA03 and by the National Natural Science Foundation of China with  grant number 12171002.}

\date{}

\maketitle

\begin{abstract}
For a  given graph   $F$, the $F$-saturation number of  a graph $G$,    denoted by $ \mathrm{sat}(G, F)$,  is the minimum  number of edges in  an  edge-maximal $F$-free subgraph of $G$.
In 2017, Kor\'andi and Sudakov determined  $ \mathrm{sat}(\mathbbmsl{G}(n, p), K_r)$ asymptotically, where $\mathbbmsl{G}(n, p) $ denotes  the Erd\H{o}s$\text{\bf--}$R\'enyi random graph and $  K_r$ is the complete graph  on $r$ vertices. In this paper, among other results,
we
present an asymptotic upper bound on $\mathrm{sat}(\mathbbmsl{G}(n, p), F)$  for any bipartite graph $F$ and also an
asymptotic lower bound on $\mathrm{sat}(\mathbbmsl{G}(n, p), F)$  for any complete  bipartite graph $F$. \\ \vspace{-2mm}

\vspace{-2mm}

\vspace{1cm}
\end{abstract}

\section{Introduction}

All graphs in this paper  are assumed to be finite, undirected,  and without loops or multiple edges.  Fix a graph $F$. We say that a graph $ G $ is {\it $ F $-free} if $ G $ has no subgraph isomorphic to $ F $. Tur\'an \cite{tur} posed one of the  foundational problems   in extremal graph theory in 1941 which was about the  maximum number
of edges in   an $ F $-free   graph on $n$ vertices. Later, Zykov \cite{zyk} introduced a dual idea  in 1949 which asks for  the minimum  number of edges in an edge-maximal $F$-free graph  on  $n$ vertices. In this paper, we deal with a random version of this concept. We below define the   concept  in a more general form.

Let $ G $ be a graph. The edge set of $G$ is  denoted by $E(G)$. A spanning subgraph  $H$  of $G$  is  said to be an {\it $F$-saturated subgraph of $G$} if $H$ is $F$-free and the addition of any edge from $E(G)\setminus E(H)$ to $H$ creates a copy of $F$. The  minimum number of edges  in  an  $F$-saturated  subgraph of $G$ is  denoted by $\mathrm{ sat}(G, F)$. Let $ K_r $ be the complete   graph on   $r$ vertices and $ K_{s,t} $ be the complete bipartite graph with parts of sizes $ s $ and $ t $.
Usually, $ \mathrm{ sat}(K_n, F)$ is   written as   $ \mathrm{ sat}(n, F)$.
Erd\H{o}s, Hajnal, and Moon     \cite{erd}   proved  that
$$\mathrm{ sat}(n,  K_r)=(r-2)n-{{r-1}\choose{2}},$$ where $n\geqslant r\geqslant 2$.
Also, with the assumption $ t\geqslant s $, Bohman, Fonoberova, and Pikhurko \cite{pikh} proved that
\begin{align*}
\mathrm{sat}(n, K_{s,t})=\frac{2s+t-3}{2}n+O\left(n^{\frac{3}{4}}\right).
\end{align*}
We refer the reader to the survey   \cite{fau} for more  known results on saturation in graphs.

Recall that  the Erd\H{o}s--R\'enyi random graph model $\mathbbmsl{G}(n, p)$ is the
probability space of all graphs on a fixed vertex set of size $n$ where every two  distinct    vertices  are adjacent   independently with probability $p$.
Throughout this paper, $ p$ is assumed to be a fixed real number in  $ (0,1)  $.
Recall that the  notion `with high probability', which is written as `{\sl whp}' for brevity,    is used whenever   an event  occurs in     $\mathbbmsl{G}(n, p)$  with a  probability approaching $1$ as $n\rightarrow\infty$.
The study of saturation numbers  in random graphs was initiated in 2017 by Kor\'andi and   Sudakov \cite{kor}.
They proved that {\sl whp}
$$\mathrm{sat}\big(\mathbbmsl{G}(n, p), K_r\big)=\big(1+o(1)\big)n\log_{\frac{1}{1-p}}n$$
for any fixed     $r\geqslant3$. Mohammadian and Tayfeh-Rezaie \cite{MT} studied the saturation numbers for stars and  found that {\sl whp}
\begin{align*}
\mathrm{sat}\big(\mathbbmsl{G}(n, p), K_{1, t}\big)=\frac{t-1}{2}n-\big(t-1+o(1)\big)\log_{\frac{1}{1-p}}n
\end{align*}
for any fixed  $t\geqslant2$. Their result   was refined by Demyanov and Zhukovskii in \cite{DZ} where  it  has been proved  that   {\sl whp} $\mathrm{sat}\big(\mathbbmsl{G}(n, p), K_{1, t}\big)$ is concentrated in a set of two  points. The related classical  result
had been proved by K\'aszonyi  and   Tuza    \cite{kas} as
\begin{align*}
\mathrm{sat}(n, K_{1, t})=\left\{\begin{array}{llll}\vspace{-4.5mm}&\\
\mathlarger{{{t}\choose{2}}+{{n-t}\choose{2}}} & \quad  \mbox{{\large if }} \mathlarger{t+1\leqslant n\leqslant\frac{3t}{2}}\mbox{,}\\
\vspace{-1mm}\\
\mathlarger{\left\lceil\frac{t-1}{2}n-\frac{t^2}{8}\right\rceil}   & \quad    \mbox{{\large if }} \mathlarger{n\geqslant\frac{3t}{2}}\mbox{.}\\\vspace{-4.25mm}&
\end{array}\right.
\end{align*}
Demidovich,  Skorkin, and Zhukovskii \cite{DSZ} proved that {\sl whp}
\begin{align*}
\mathrm{sat}\big(\mathbbmsl{G}(n, p), C_k\big)=n+\mathnormal{\Theta}\left(\frac{n}{\log n}\right)
\end{align*}
for any $ k\geqslant 5$, where $ C_k$ is a cycle graph on $ k$ vertices, while
$$
\left(\frac{3}{2}+o(1)\right)n\leqslant\mathrm{sat}\big(\mathbbmsl{G}(n, p), C_4\big)\leqslant  \big(c_p+o(1)\big)n
$$
for some explicit constant $c_p$. In particular,   $c_{1/2}=27/14$.

The exact values of both $ \mathrm{sat}(n, K_{s, t}) $ and $ \mathrm{sat}(\mathbbmsl{G}(n, p), K_{s,t}) $ are still unknown. Note that, for any  connected graph $F$ with no cut edges, both $\mathrm{sat}(n, F)$ and $ \mathrm{sat}(\mathbbmsl{G}(n, p), F) $ are at least $n-1$, since each    $F$-saturated subgraph should be connected. Therefore, in particular, {\sl whp}  $ \mathrm{sat}(\mathbbmsl{G}(n, p), K_{s,t})\geqslant  n-1$   if $t\geqslant  s\geqslant  2$. Diskin,   Hoshen, and  Zhukovskii  \cite{DHZ}  showed    that,     for any bipartite graph $F$, there exists a constant $c_F$ such that  $ \mathrm{sat}(\mathbbmsl{G}(n, p), F)\leqslant  c_Fn$ {\sl whp}. However,  an explicit value of $c_F$ was not known.
In this paper, we prove the following theorem for any  arbitrary bipartite
graph.

\begin{theorem}\label{upb}
Let $ p\in (0, 1)$ be constant and let $ F$ be a bipartite graph with no  isolated vertices. Let $\{A_1, B_1\},  \ldots, \{A_k, B_k\}$ be the vertex bipartitions of all the connected components of $ F$ with
$ |B_i|\geqslant |A_i|$ for every $ i$. Let $ a=\max\{|A_1|, \ldots, |A_k|\}$ and $\delta$ be the minimum degree over all vertices from $A_i$ with $|A_i|=a$. Then, {\sl whp}
\[
\mathrm{sat}\big(\mathbbmsl{G}(n, p), F\big)\leqslant\left(\frac{\delta -1}{p^{a-1}}-\frac{\delta -2a+1}{2}+o(1)\right)n.
\]
\end{theorem}

Our proof of Theorem  \ref{upb},  which  is presented in Section  \ref{sc:upper_bip},   is based on the construction suggested in  \cite{DHZ}. Actually, we have tuned the parameters of the construction in order to achieve the optimal bound.
For $F=K_{s, t}$  with  $ t\geqslant s$, Theorem \ref{upb} shows that {\sl whp}
$$
\mathrm{sat}(\mathbbmsl{G}(n, p), K_{s,t})\leqslant\left(\frac{t-1}{p^{s-1}}- \frac{t-2s+1}{2}+o(1)\right)n.
$$
For a lower bound on $\mathrm{sat}(\mathbbmsl{G}(n, p), K_{s,t})$,  we prove the following theorem.

\begin{theorem}\label{th:bip_intro}
Let   $ t\geqslant s\geqslant 2 $ be fixed integers and let $ p\in (0, 1)$ be constant. Then, {\sl whp}
$$
\mathrm{sat}(\mathbbmsl{G}(n, p), K_{s,t})\geqslant\left(\max\left\{ \frac{2s+t-3}{2}, \frac{t-s}{4p^{s-1}}+\frac{s-1}{2}\right\} +o(1)\right)n.
$$
\end{theorem}

The proof of the lower bound in Theorem \ref{th:bip_intro} is the most involved part of the paper. It is presented in Section  \ref{sc:lower_bip}.
For every fixed    $t>s$, our bounds in Theorem \ref{th:bip_intro}   imply that  the $K_{s,t}$-saturation number in $\mathbbmsl{G}(n,p)$ is $\mathnormal{\Theta}(p^{1-s}n)$.   Let us also note that, in  the   case $s=t=2$,   Theorem \ref{th:bip_intro} provides   the lower bound  obtained  in  \cite{DSZ}, while our upper bound  is slightly worse.

As we saw above, {\sl whp}
$ \mathrm{sat}(\mathbbmsl{G}(n, p), K_r)\gg\mathrm{sat}(n, K_r) $ for any $r\geqslant 3$. For complete bipartite graphs, the saturation number is more stable, that is,    $ \mathrm{sat}(\mathbbmsl{G}(n, p), K_{s,t})$ is linear in $n$ {\sl whp}  as well as $ \mathrm{sat}(n, K_{s,t})$. For $t>s\geqslant  2$ and   sufficiently small $p\in(0, 1)$, there is no asymptotical stability,  that is,    there exists a constant $c>1$ such that   $ \mathrm{sat}(\mathbbmsl{G}(n, p), K_{s,t})\geqslant c \, \mathrm{sat}(n, K_{s,t})$ {\sl whp}. However, for $s=t$ or $t>s\geqslant  2$ and   sufficiently large $p\in(0, 1)$, we do not know whether there is     an asymptotical stability. Finally, the $K_{1,t}$-saturation number is asymptotically stable, while  $\mathrm{sat}(\mathbbmsl{G}(n, p), K_{1,t})< \mathrm{sat}(n, K_{1,t})$ {\sl whp}.
Note that, for cycles, {\sl whp}
$\mathrm{sat}(\mathbbmsl{G}(n, p), C_k)\leqslant((k+2)/(k+3)+o(1))\mathrm{sat}(n, C_k)$ for any   $k\geqslant 5$  by  a result of  F\"uredi and  Kim     \cite{Furedi}.

For the sake of completeness,  we give in Section  \ref{sc:lower_general}  two simple general lower bounds on $\mathrm{sat}(\mathbbmsl{G}(n, p), F)$ for any  arbitrary graph $F$ which are asymptotically tight for certain graph families.

\section{Notation and preliminaries}\label{prelim}

In this section, we introduce notation   and formulate several properties of random graphs that will be used in the rest of the  paper.
First, let us fix some more notation and terminology  of  graph  theory.

Let $G$ be a graph. The vertex set of $G$ is  denoted by $V(G)$  and the {\it order}  of $G$ is defined  as $|V(G)|$.
For a  subset   $X$ of $V(G)$, we denote the induced subgraph of $G$ on $X$  by $G[X]$.
For a  subset   $Y$ of $E(G)$, we denote by $G-Y$  the   graph obtained from  $G$ by removing the edges in  $Y$.
For a subset     $Z$ of $V(G)$, set       $N_G(Z)=\{v\in V(G) \, | \, v \text{   is adjacent to all vertices in } Z\} $.  For the sake of convenience, we write $N_G(z_1, \ldots, z_k)$ instead of $N_G(\{z_1, \ldots, z_k\})$.
For a vertex $ v$ of $ G$, we define the   {\it degree}  of $v$    as $|N_G(v)|$ and denote by $ d_G(v)$. The maximum and the minimum degree of vertices of $ G$ are  denoted by $ \mathnormal{\Delta}(G)$ and $ \delta(G)$, respectively. For two subsets $ S$ and $ T$ of $ V(G)$, we denote by $ E_{G}(S, T)$ the set of all edges
with endpoints in both $ S$ and $ T$. We write $ E_G(S)$ for  $E_G(S,S)$. We drop subscripts  if there is no danger of confusion.

In what follows, we   recall the  probabilistic results  that we make use of all in the  next sections.
The next lemma is well known and can be deduced  from the Chernoff  bound     \cite[Theorem 2.1]{JLR}.

\begin{lemma}\label{Cher.Ineq.}
Let $ X \thicksim \mathrm{Bin} (n, p) $ be a binomial random variable with parameters $ n$
and $ p$. If $ \mathbbmsl{E}[X]\rightarrow \infty $ as $ n \rightarrow \infty $, then
$ X=\mathbbmsl{E}[X](1+o(1))$ {\sl whp}.
\end{lemma}

The following  lemma is a consequence of  Proposition 19 in \cite{Bottcher}.

\begin{lemma}\label{bandwidth}
For any constant $ p\in (0, 1)$, there is a constant $c$, depending on $p$, such that
$\mathbbmsl{G}(n, p)$   has the following property {\sl whp}.
For every  subset $ X$ of vertices with $|X|\geqslant c\log n$, the number of vertices with no neighbors in
$X$ is at most $c\log n $.
\end{lemma}

The following Lemma is an immediate consequence of the Chernoff  bound     \cite[Theorem 2.1]{JLR} and the union bound.

\begin{lemma}\label{random1}
Let $\lambda>1 $ and $ p\in (0, 1)$ be constants. Then,  $ \mathbbmsl{G}(n, p)$ has the following property {\sl whp}.
For every two disjoint subsets $ X, Y $ of vertices of size at least $ \log^{\lambda}n$, we have
$ |E(X)|=p\binom{|X|}{2}(1+o(1))$ and $ |E(X, Y)|=p|X||Y|(1+o(1))$.
\end{lemma}

The following corollary  follows from  Lemma \ref{random1}   immediately.

\begin{corollary}\label{cor.random}
Let $ \lambda>1 $ and $ p\in (0, 1)$ be constants. Then,  $ \mathbbmsl{G}(n, p)$  has the following property {\sl whp}.
For every two subsets $ X, Y $ of vertices,    $ |E(X, Y)|\leqslant 3n\log^{\lambda}n+p|X||Y|(1+o(1))$.
\end{corollary}

Note that,  for a   positive fixed  integer $M$,  the probability that $\mathbbmsl{G}(n, p)$ does not contain a clique of size $M$ is $\exp(-\mathnormal{\Theta}(n^2))$ due to the  Janson bound  \cite[Theorem 2.14]{JLR}. Therefore, by the union bound, we get the following.

\begin{lemma}\label{random2}
Let $ \lambda>1 $,   $ p\in (0, 1)$  be constants  and let  $M$ be a  positive fixed integer. Then,  $ \mathbbmsl{G}(n, p)$  has the following property {\sl whp}.
Every subset $ X $ of vertices of size at least $ \log^{\lambda} n$ contains a clique of size $ M$.
\end{lemma}

\section{General lower bounds}\label{sc:lower_general}

In this section, we prove a lower bound on $ \mathrm{ sat}(G, F)$ for every two graphs $ G$ and $ F$ which provides a lower bound on $ \mathrm{sat}( \mathbbmsl{G}(n,p), F) $.
It is trivial that
\[
\mathrm{sat}(G, F)\geqslant \frac{\min\{\delta(G), \delta(F)-1\}}{2}n
\]
for any two graphs $ G$ and $ F$. In order to proceed,   we need the following definition.

Let $G$ be a graph and $k$ be a nonnegative   integer. A subset $S$ of $V(G)$  is called {\it $k$-independent}
if the maximum degree of $G[S]$ is at most  $k$. The  {\it $k$-independence number}  of $G$, denoted
by $\alpha_k(G)$, is defined as the maximum cardinality of a $k$-independent set in  $G$. In particular, $\alpha_0(G)=\alpha(G)$ is the usual independence
number of $G$.
Furthermore, define $r(G)=\min_{xy\in E(G)}\max \{ d(x) , d(y)\} $.

\begin{theorem}\label{anyg}
Let $ F $ be  a graph and let $r=r(F)$.
If $ r\geqslant 2$, then, 	for every graph $G$ on $n$ vertices,
$$\mathrm{ sat}(G, F)\geqslant\frac{(r-1)\big(n-\alpha_{r-2}(G)\big)}{2}.$$
\end{theorem}

\begin{proof}
Let $H$ be an   $ F$-saturated   subgraph of   $G$. Let $ r\geqslant 2$ and let $A$ be the set of vertices of $H$  with   degree  at  most $r-2$ in $H$. Suppose that there are two vertices $x, y\in A$ with $ xy\in E(G)\setminus E(H)$. By definition of $ r$, adding $ xy$ to $ H$ does not create a copy of $ F$. This is a contradiction, since $H$ is  an  $F$-saturated subgraph of $G$.
This implies that $G[A]=H[A] $ and so  $|A|\leqslant\alpha_{r-2}(G)$. We hence obtain that
\begin{equation*}
|E(H)|\geqslant\dfrac{\sum\limits_{v\in V(H)\setminus A}d_H(v)}{2}\geqslant\frac{(r-1)\big(n-\alpha_{r-2}(G)\big)}{2}.
\qedhere
\end{equation*}
\end{proof}

\begin{theorem}[\cite{FKM, MT}]\label{k}
For every constants $p\in(0, 1)$ and $k\geqslant 0$, {\sl whp}
$$\alpha_k\big(\mathbbmsl{G}(n, p)\big)=\big(2+o(1)\big)\log_{\frac{1}{1-p}}n.$$
\end{theorem}

Actually, we  known  from  \cite{KSZ} that  $\alpha_k(\mathbbmsl{G}(n, p))$ is concentrated in a set of two  consecutive points {\sl whp}. Using Theorems \ref{anyg} and \ref{k}, we conclude the following.

\begin{corollary}\label{cor}
Let $ F $ be  a graph and let $r=r(F)$. Then,  for each fixed real number  $ p\in (0,1)$, {\sl whp}
\[
\mathrm{sat}\big(\mathbbmsl{G}(n, p), F\big)\geqslant \frac{r -1}{2}n-\big(r-1+o(1)\big)\log_{\frac{1}{1-p}}n.
\]
\end{corollary}

For  $F=K_{1,t}$,     the lower  bound  given in  Corollary \ref{cor}  is tight by a result in  \cite{MT}. However,  for graphs     $F$   satisfying  the property that  each  edge   $uv\in E(F)$  with   $\max\{ d(u), d(v)\}=r(F)$  is contained  in  a   triangle, the lower bound can be significantly improved.

For any graph $G$, define  $w(G)=\min_{xy\in E(G)}\{\max\{ d(x) , d(y)\}+|N(x)\cap N(y)|\}$.   Cameron and  Puleo \cite{Cam} proved that
$$\mathrm{ sat}(n, F)\geqslant\frac{w(F)-1}{2}n-\dfrac{ w(F)^2-4 w(F)+5}{2}$$
for any $ n$. Below, we  give a lower bound on $ \mathrm{sat}(\mathbbmsl{G}(n, p), F) $ in terms of $  w(F)$ which is asymptotically stronger than Corollary \ref{cor} for many graphs $ F$.

\begin{theorem}\label{thm:weight}
For any constant $ p\in (0,1)$ and any graph $ F $, {\sl whp}
$$\mathrm{sat}\big(\mathbbmsl{G}(n, p), F\big)\geqslant\frac{ w(F)-1}{2}n-O(\log n).$$
\end{theorem}

\begin{proof}
If $  w(F)=1$, then there is nothing to prove.  So, assume that $  w(F)\geqslant 2$.
Let $ G\thicksim\mathbbmsl{G}(n, p)$ and $ \ell= c\log n$, where $ c$ is given in Lemma \ref{bandwidth}.
Assume that $ H$ is an arbitrary $ F$-saturated subgraph of $ G$ whose vertices are labeled as
$ u_1, \ldots, u_n$ for which $d_H(u_1)\leqslant \cdots \leqslant d_H(u_n) $. Let $ U=\{u_1, \ldots, u_{\ell}\}$. For $ i=1, \ldots, \ell$, let $ V_i=N_H(u_i)$ and $ V=\bigcup_{i=1}^{\ell}V_i$. Also, for $ i=1, \ldots,  \ell $, define $ W_i= N_G(u_i)\setminus (U\cup V\cup W_1\cup \cdots \cup W_{i-1} ) $ and set $ W= \bigcup_{i=1}^{\ell}W_i$.	
If $ d_H(u_{\ell})\geqslant w(F)-1 $, then
\[
|E(H)|\geqslant
\frac{\sum\limits_{i=\ell+1}^{n}d_H(u_i)}{2}\geqslant \frac{\big( w(F)-1\big)(n-\ell)}{2}
\]
which concludes the assertion. So, we may assume that $ d_H(u_{\ell})\leqslant w(F)-2 $. Then,
$ |V|\leqslant \sum_{i=1}^{\ell}d_H(u_i) \leqslant \ell( w(F)-2) $.
Let $ R= V(G)\setminus (U\cup V \cup W )$. Note that $ R$ is the set of all vertices in $ V(G)\setminus U$ which are not adjacent to any vertex in $ U$ and so $ |R|\leqslant c\log n$ by  Lemma \ref{bandwidth}.
Let   $x\in W_i$ and let $ F'$ be a copy of $ F$ in  $ H+xu_i$. It follows from $ d_H(x)\geqslant d_H(u_i)$ that $ d_H(x) \geqslant \max \{d_{F'}(x), d_{F'}(u_i)\}-1$. Since $ N_{H}(x)\cap V\supseteq N_{F'}(x)\cap N_{F'}(u_i)$, one concludes  that
\[
d_H(x)+|N_{H}(x)\cap V|\geqslant \max \big\{d_{F'}(x), d_{F'}(u_i)\big\}-1+|N_{F'}(x)\cap N_{F'}(u_i)| \geqslant   w(F)-1.
\]
Now, we may write
\begin{align}
2|E(H)|&\geqslant \sum_{x \in V} d_H(x)+\sum_{x \in W} d_H(x)\nonumber\\
&\geqslant \sum_{x \in V}|N_H(x)\cap U|+\sum_{x \in V}|N_H(x)\cap W|+\sum_{x \in W} d_H(x)\nonumber\\
&\geqslant |V|+\sum_{x \in W} |N_H(x)\cap V|+ \sum_{x \in W} d_H(x)\nonumber\\
&=|V|+\sum_{x \in W} \big(d_H(x)+|N_H(x)\cap V|\big)\nonumber\\
&\geqslant |V|+\sum_{x \in W} \big( w(F)-1\big)\nonumber\\
&= |V|+\big( w(F)-1\big)\big(n-|U|-|V|-|R|\big)\nonumber\\
&=\big( w(F)-1\big)n-\big( w(F)-2\big)|V|-\big( w(F)-1\big)\big(\ell +|R|\big) \label{Ineq}
\end{align}
Since $ \ell = c\log n$, $ |V|\leqslant \ell( w(F)-2) $, and $ |R|\leqslant c\log n$,
the result follows from \eqref{Ineq}.
\end{proof}

\section{Upper bound for bipartite graphs}\label{sc:upper_bip}

In this section,  we prove Theorem  \ref{upb}. Our proof is based on the construction suggested in  \cite{DHZ} which, in turn, resembles the proof strategy of a general linear in $n$ upper bound on $\mathrm{sat}(n, F)$ from  \cite{kas}.
First, we present a useful observation which can be proved straightforwardly.

\begin{observation}\label{obs}
Let $ H$ be an $ F$-free subgraph of  $ G$. Then,  there is an $ F$-saturated subgraph of $ G$ which has $ H$ as a subgraph.
\end{observation}

Below,  we show    how a general linear in $n$ upper bound on $\mathrm{sat}(n, F)$ can be derived from Observation \ref{obs}. While we use the same construction as in  \cite{kas}, we formulate the proof in a different way in order to make the move to random settings smoother.

\begin{theorem}[\cite{kas}]\label{th:known}
Let $ F$ be a graph and $ S $ be an independent set in $ F$ with maximum possible size. Let
$ b = |V(F)|-|S|-1$ and
$ d= \min\{|N_F(x)\cap S| \,|\, x\in V(F)\setminus S\}$. Then,
\[
\mathrm{sat}(n, F)\leqslant \frac{2b +d-1}{2}n-\frac{b(b+d)}{2}.
\]
\end{theorem}

\begin{proof}
Let $B$ be a subset of $V(K_n)$ of  size $b$ and let  $\overline{B}=V(K_n)\setminus B$. Consider the spanning subgraph $H_0$ of $ K_n $ obtained by deleting all edges whose both endpoints are  in  $\overline{B}$. If there is a copy $ F'$ of $ F$ in $ H_0$, then $ V(F')\cap \overline{B}$ is an independent set of size
\[
|V(F')\cap \overline{B}|= |V(F')|-|V(F')\cap B|\geqslant |V(F)|-|B|=|S|+1,
\]
a contradiction. This shows that $ H_0$ is $ F$-free.  Using Observation \ref{obs}, there is an $ F$-saturated subgraph of $ G$, say $ {H}$,  with $ E(H)\supseteq E(H_0)$.
For every  $x\in  \overline{B}$, we have $ |N_{H}(x)\cap \overline{B}|\leqslant d-1$, as otherwise the  subgraph of $ H $ with the edge set $ E(H_0)\cup E_H(\{x\}, \overline{B})$ contains a copy of $ F$.
Since
\[
|E(H)|=|E(H_0)|+\sum_{x \in \overline{B}}|N_{H}(x)\cap \overline{B}|\leqslant |E(H_0)|+\frac{(d-1)|\overline{B}|}{2},
\]
the result follows.
\end{proof}

Let us now prove Theorem  \ref{upb}. Note that it is impossible to find a construction as in the proof of Theorem  \ref{th:known},  since vertex  degrees in the random graph equal $np(1+o(1))$. Thus, instead of considering a single clique $B$ with its common neighborhood, we will consider $\mathnormal{\Theta}(\ln n)$ disjoint sets of constant sizes as well as their common neighborhoods. For the sake of  convenience, we handle the case of $F$ being a disjoint union of stars  separately. This proves Theorem \ref{upb}  for  the  case   $a=1$ and  generalizes  a   result  given  in \cite{MT}.

\begin{lemma}\label{lm:star}
Let $ p\in (0, 1)$ be constant and let $ F$ be the disjoint union of stars $ K_{1, t_1}, \ldots, K_{1, t_k}$ with $ k\geqslant 1 $ and $ t_1\geqslant \cdots \geqslant t_k\geqslant 1 $. Then, {\sl whp}
$$\mathrm{sat} \big(\mathbbmsl{G}(n, p), F\big)=\frac{t_k-1}{2}n-\big(t_k-1+o(1)\big)\log_{\frac{1}{1-p}}n.$$
\end{lemma}

\begin{proof}
In view of Corollary \ref{cor}, it suffices to prove the upper bound. Using Theorem \ref{k},
$  \alpha(\mathbbmsl{G}(n, p))=(2+o(1))\log_{1/(1-p)}n$  {\sl whp}. Let $G\thicksim\mathbbmsl{G}(n, p)$ and $h=|V(F)|-1 $.
Fix an integer-valued function $\ell=\ell(n)=(2+o(1))\log_{1/(1-p)}n$ such that  $(n-h-\ell)(t_k-1)$ is even and $\alpha(\mathbbmsl{G}(n, p))\geqslant \ell$  {\sl whp}.
Also, let $ L$ be the disjoint union of $ K_h$ and an arbitrary regular graph on   $n-h-\ell$ vertices with   degree  $t_k-1$.
We know from  a result of Alon  and F\"uredi \cite{AF} that,  for sufficiently small $\varepsilon>0$, the graph $\mathbbmsl{G}(n-\ell, n^{-\varepsilon})$ contains a copy of $L$ {\sl whp}. Using the standard multiple-exposure technique, it implies that $\mathbbmsl{G}(n-\ell, p)$ does not contain  a copy of $L$ with probability at most $\exp(n^{-\varepsilon+o(1)})$. Thus, by the union bound, {\sl whp} there exists a subset  $S\subseteq V(G)$ with $|S|=\ell$ such that $S$ is an  independent set in $G$  and $G[V(G)\setminus S]$ has a copy $L'$  of  $ L$ as a subgraph. Denote by $H$  the spanning subgraph of $G$ with  the edge set $E(L')$. It is easily seen that $H$ is an $F$-saturated subgraph of $G$ and
$$|E(H)|=\frac{(n-h-\ell)(t_k-1)}{2}+\binom{h}{2}=\frac{t_k-1}{2}n-\big(t_k-1+o(1)\big)\log_{\frac{1}{1-p}}n$$
which  concludes the result.
\end{proof}

\begin{remark}
Note that  Lemma  \ref{lm:star} for $ t_k=1$ could be strengthened as follows.
If  $ F$ is a graph with a connected component $ K_2$, then
$ \mathrm{sat} (\mathbbmsl{G}(n, p), F) \leqslant \binom{|V(F)|-1}{2}$ {\sl whp}. Conversely, if $ \mathrm{sat} (\mathbbmsl{G}(n, p), F)$ is bounded from above by a constant, then Corollary \ref{cor} forces $ F$ to have a connected component $ K_2$.
\end{remark}

\begin{proof}[Proof of   Theorem  \ref{upb}]
In view of  Lemma  \ref{lm:star}, we may  assume that
$ a\geqslant 2$.
Let $ G\thicksim\mathbbmsl{G}(n, p)$, $ b=1-p^{a-1}$, and $  \ell = \lfloor\log_{1/b}n^{2/3}\rfloor$.
Without loss of generality, assume that $ |A_1|= \cdots = |A_q|>|A_{q+1}|\geqslant \cdots \geqslant |A_k|$ for some $ q$.	Fix disjoint arbitrary $ (a-1) $-subsets $ V_1,  \ldots, V_{\ell}$ and $ (a+1) $-subsets $ V_{\ell +1},  \ldots, V_{\ell +q-1}$ of $ V(G)$.  Set $ V= \bigcup_{i=1}^{\ell}V_i$ and $ V'= \bigcup_{i=\ell +1}^{\ell+q-1}V_i$. Let $ M_i =\bigcup_{j=1}^{i} N(V_j)$ for any $ i\geqslant 1$.
For $ i=1, \ldots,  \ell+q-1 $, define $ W_i= N(V_i)\setminus (V\cup V' \cup M_{i-1} ) $ and set $ W= \bigcup_{i=1}^{\ell}W_i$.
Let $ R= V(G)\setminus (V \cup W )$. Note that $ R= (V(G)\setminus (V\cup M_{\ell}))\cup V'$.
Set $ V''=V'\cap M_{\ell}$.
A   schematic of    the  structure of $V(G)$  is illustrated in Figure \ref{fg:fg}.

\begin{figure}[H]
\begin{center}
\includegraphics[width=400pt,height=200pt]{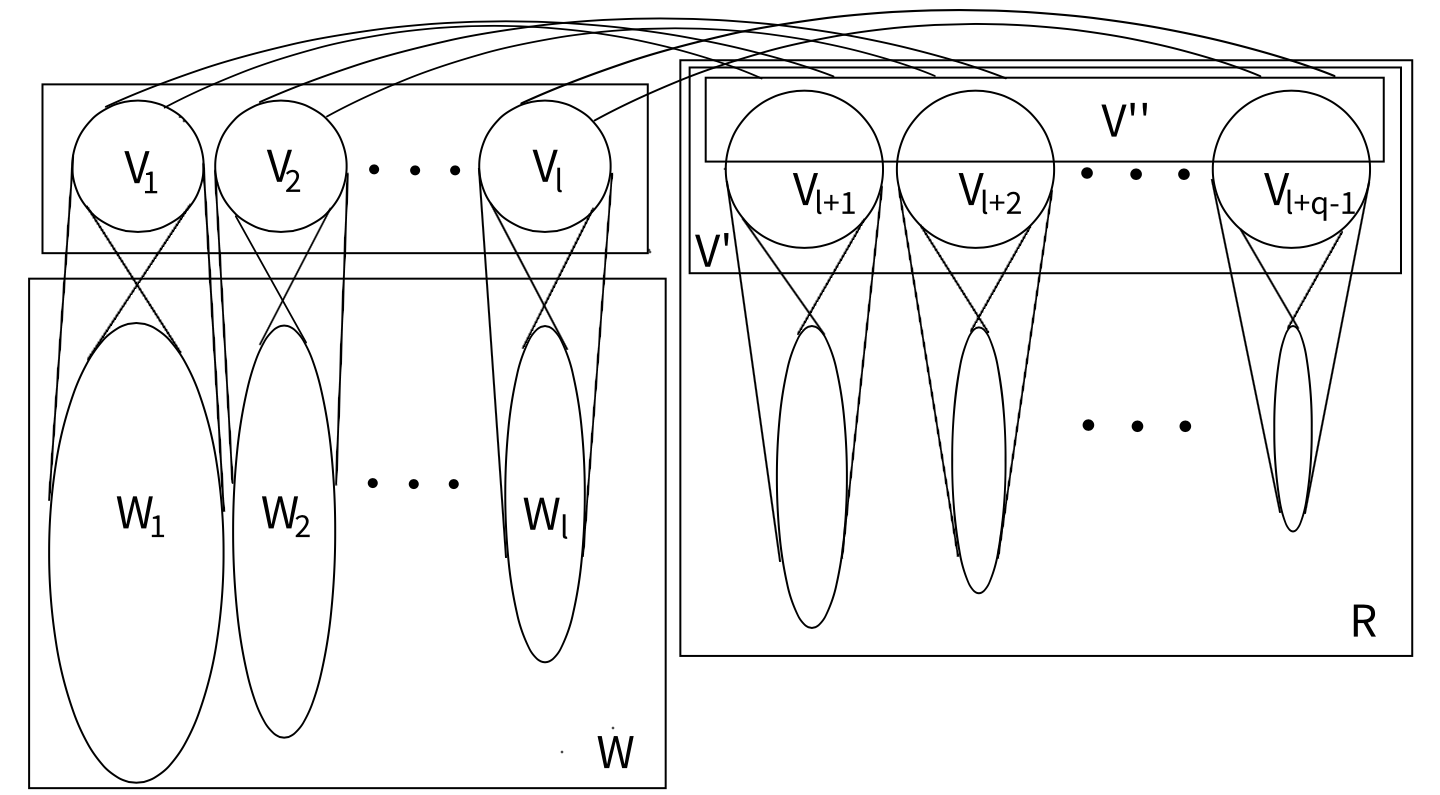}
\caption{The     structure of $V(G)$     described in the proof of Theorem \ref{upb}.}
\label{fg:fg}
\end{center}
\end{figure}

As $ |R\setminus V''| \thicksim \mathrm{Bin}(n-\ell(a-1), b^{\ell})$, Lemma \ref{Cher.Ineq.} implies that {\sl whp}
\[
|R|= b^{\ell}\big(n-\ell(a-1)\big)\big(1+o(1)\big)+|V''|
\]
which gives that $ |R|=O(n^{1/3})$. Similarly,  for all $i$, {\sl whp}
\[
|W_i|= b^{i-1}(1-b)\big(n-\ell(a-1)-(q-1)(a+1)\big)\big(1+o(1)\big)
\]
which yields that $ |W_i|=\mathnormal{\Omega} (n^{1/3})$ for $ i=1, \ldots, \ell+q-1 $. In particular, $ |W_i|\geqslant \max\{|B_1|, \ldots, |B_k|\}+1$.

Let $ H_0$ be a spanning subgraph of $ G$ with $ E(H_0)=\cup^{\ell+q-1}_{i=1}E_G(V_i, W_i)$.
By the definition of $ a$, we conclude that $ H_0$ is $ F$-free. Using Observation \ref{obs}, there is an $ F$-saturated subgraph of $ G$, say $ {H}$,  with $ E(H)\supseteq E(H_0)$. Now, we bound the number of edges of $ {H}$.
We will use
\begin{align}\label{eq1}
|E(H)|&= |E_H(V(G)\setminus W)|+|E_H(V(G)\setminus W, W)| + |E_H(W)|.
\end{align}
It follows	from $V(G)\setminus W=R\cup V $ that   $ |V(G)\setminus W|= O(n^{1/3})$ and hence
$ |E_H(V(G)\setminus W)|= O(n^{2/3}) $.
For every $ i\in \{1, \ldots, \ell\} $ and every $x\in  V(G)\setminus V_i$, we have $ |N_{H}(x)\cap W_i|\leqslant \delta-1$, as otherwise the bipartite subgraph of $ H $ with the edge set $ E(H_0)\cup E_H(\{x\}, W_i)$ contains a copy of $ F$.
Therefore,
\begin{align}\label{eq2}
|E_H( V(G)\setminus W, W)|&=\sum_{i=1}^{\ell} |E_H( V(G)\setminus W, W_i)|\nonumber\\
&=\sum_{i=1}^{\ell} |E_H( V(G)\setminus (V_i\cup W ), W_i)|+\sum_{i=1}^{\ell}|E_H(V_i, W_i)|\nonumber\\
&\leqslant  \ell(\delta-1)|V(G)\setminus W|+ \sum_{i=1}^{\ell}|E_H(W_i, V_i)|\nonumber \\
&\leqslant O\left(n^{\frac{1}{3}}\log n\right)+(a-1)n.
\end{align}
It remains to estimate $ |E_H( W)|$. To do this, we write
\begin{align}\label{eq3}
|E_H( W)|&= \sum_{i=1}^{\ell}\sum_{j=1}^{i-1}|E_H(W_i, W_j)|+ \sum_{i=1}^{\ell}|E_H(W_i)|\nonumber\\
&\leqslant \sum_{i=1}^{\ell}(i-1)(\delta -1)|W_i|+\sum_{i=1}^{\ell}\frac{\delta -1}{2} |W_i|\nonumber\\
&=\frac{\delta -1}{2} \sum_{i=1}^{\ell}(2i-1)|W_i|\nonumber\\
&\leqslant \frac{\delta -1}{2} \sum_{i=1}^{\ell}(2i-1)b^{i-1}(1-b)n\big(1+o(1)\big)\nonumber\\
&\leqslant \frac{\delta -1}{2}(1-b)n\big(1+o(1)\big) \sum_{i=1}^{\ell}(2i-1)b^{i-1}\nonumber\\
&= \frac{\delta -1}{2}(1-b)n\big(1+o(1)\big) \frac{1+b-(2\ell +1)b^{\ell}+(2\ell -1)b^{\ell+1}}{(1-b)^2}\nonumber\\
& \leqslant \frac{\delta -1}{2} \left(\frac{1+b}{1-b}\right)n\big(1+o(1)\big).
\end{align}
By \eqref{eq1}--\eqref{eq3}, we conclude that
\begin{align*}
|E(H)|&\leqslant \left( \frac{\delta -1}{2} \left(\frac{1+b}{1-b}\right)+a-1\right)n\big(1+o(1)\big)\\
&=  \left(\frac{\delta -1}{p^{a-1}}-\frac{\delta -2a+1}{2}+o(1)  \right)n.
\end{align*}
Since $ \mathrm{ sat}(G, F)\leqslant |E(H)|$, the result follows.
\end{proof}

\section{Lower bound for $\boldsymbol{ K_{s,t}}$}\label{sc:lower_bip}

In this section,  we prove the two lower bounds in Theorem  \ref{th:bip_intro}. We start from the bound that does not depend on $p$  that is  stated separately below. Let us recall that this bound generalizes the lower bound from  \cite{DSZ} for $F=K_{2,2}$, that is,
$\mathrm{sat}(\mathbbmsl{G}(n, p), K_{2,2}) \geqslant (\frac{3}{2}+o(1))n$ {\sl whp}.  However, our argument
is simpler and resembles the argument used by Bohman, Fonoberova, and  Pikhurko \cite{pikh} for their asymptotic lower bound on $ \mathrm{sat}(n, K_{s,t})$.

\begin{theorem}\label{pikh.low}
Let $ t\geqslant s\geqslant 2 $ and $ p\in (0, 1)$ be  constants. Then, {\sl whp}
\[
\mathrm{sat}\big(\mathbbmsl{G}(n, p), K_{s,t}\big)\geqslant \left(\frac{2s+t-3}{2}+o(1)\right)n.
\]
\end{theorem}

\begin{proof}
Let $ G\thicksim\mathbbmsl{G}(n,p) $ and let $ H $ be a $ K_{s,t} $-saturated subgraph of $ G$ with minimum possible number of edges. Let $V=V(G) $. By Theorem \ref{upb},  we have that  $ |E(H)| = O(n) $ {\sl whp}. For   the subsets
\begin{align*}
A &= \left\{x\in V  \,   \left| \, d_H(x) \geqslant n^{\frac{1}{4}}\right.\right\},\\
B &= \big\{x \in V  \setminus A \, \big| \, |N_H(x)\cap A| \leqslant s - 2\big\},\\
C &= \big\{x \in V  \, \big| \, d_H(x) \leqslant s+t - 3\big\},\\
D& = V \setminus (A\cup B \cup C).
\end{align*}
of $ V(H)$, we   prove  the following claims.

\begin{claim}\label{cl:1}
{\sl Whp} $ |A|=O(n^{3/4})$.
\end{claim}

\begin{proof}
Since $ |E(H)| \geqslant |A| n^{1/4}/2 $, we have
$|A| = O(n^{3/4})$.
\renewcommand{\qedsymbol}{$\da$}
\end{proof}

\begin{claim}\label{cl:2}
{\sl Whp}  $ |B|=O(n^{3/4})$.
\end{claim}

\begin{proof}
Take any two  vertices $ x, y \in B $ such that $\{x,y\}\in E(G)\setminus E(H) $. The addition of $ xy $
to $ H $ creates a copy of $ K_{s,t} $ with vertex bipartition $ \{ X,  Y \}$ so that $ x \in X$ and
$ y \in Y $. Since $ |N_H(x)\cap A| \leqslant s-2$, $x $ has a neighbor $ y'\in Y \setminus A$. Similarly, $ y$
has a neighbor $ x'\in X \setminus A$. This shows that there is a path  $ x, y', x', y$ of length three which connects $ x$ to $ y$ in $ H- A$.
Therefore, every two vertices in $B$ which are adjacent in $G$ are connected in $ H-A$ by a path of length one or three.
If $|B|\leqslant n^{3/4} $, then we are done. Otherwise, by Lemma \ref{random1}, we have
\[
\frac{p}{2}\binom{|B|}{2} \leqslant |E(G[B])|\leqslant |E(H[B])|+|B|\left(n^{\frac{1}{4}}n^{\frac{1}{4}}n^{\frac{1}{4}}\right)
\leqslant \frac{|B|n^{\frac{1}{4}}}{2}+|B|n^{\frac{3}{4}}
\]
which gives $ |B|=O(n^{3/4})$.
\renewcommand{\qedsymbol}{$\da$}
\end{proof}

\begin{claim}\label{cl:3}
$|C|\leqslant\log^{2} n$.
\end{claim}

\begin{proof}
By contradiction, assume that $ |C|>\log^{2} n$.
Recall that the Ramsey number
$ R_{s+t-3}(s+t) $ is the smallest positive integer $ m $ such that any coloring of the edges of $ K_m $ with
$s+t-3$
colors gives a monochromatic copy of $ K_{s+t} $.
Using Lemma \ref{random2}, $ C $ contains a clique $ C' $ of size $M= (s+t-2)R_{s+t-3}(s+t) $ in $ G$. We know that every graph $ \mathnormal{\Gamma}$ contains an independent set of size at least $ |V(\mathnormal{\Gamma})|/(\mathnormal{\Delta}(\mathnormal{\Gamma})+1)$.
Since each
vertex of $ C' $ has degree at most $ s+t-3 $ in $ H$, there is an independent set $ C''\subseteq C'  $ with $ |C''|\geqslant R_{s+t-3}(s+t) $ in $H $.
For each vertex
$ x \in C'' $, fix an arbitrary ordering of $ N_H(x) $ which we encode
by a bijection $ f_x : N(x) \rightarrow \{1, \ldots, d_{H}(x)\} $.
For each pair of distinct vertices $ x, y \in C'' $
do the following. Fix a copy
of $ K_{s,t} $ in $ H+xy $ with partition $ \{X, Y\} $ so that $ x \in X $ and $ y \in Y$ . Since
$|(X\setminus\{x\}) \cup (Y\setminus\{y\})|=s+t-2$, there are $ x'\in X\setminus\{x\}  $ and $ y'\in Y \setminus\{y\}$ with $ f_x(x')=f_y(y')$. Denote the integer $ f_x(x')=f_y(y')$ by $ a$.  Clearly, $ x'y'\in E(H) $. Now, color the edge $ xy$ by $ a$.
This defines an edge coloring of $ E(G[C''])$ with
$ s+t-3 $ colors. By Ramsey's theorem, there is a  $ (s+t) $-subset $ C'''\subseteq C'' $ such that all edges of $G[C'''] $ have the same color, say $ c $. For every two distinct vertices $ x,y\in C'''$,  as  $ f^{-1}_x (c)$ and $ f^{-1}_y (c)$ are adjacent in $ H$,  $ f^{-1}_x (c)\neq f^{-1}_y (c)$.   So $\{f^{-1}_x (c) | x\in C''' \}  $ is a clique of order $ s+t $ in $ H $ which contradicts the $ K_{s,t} $-freeness of $ H $, proving the claim.
\renewcommand{\qedsymbol}{$\da$}
\end{proof}

Using Claims  \ref{cl:1}--\ref{cl:3}, we conclude that $ |D|=n-O(n^{3/4})$.
Since every vertex in $ D$ has at least $ s-1$ neighbors in $ A$, we may choose $ s-1$ distinct edges for each vertex of $ D $. Put all these edges in a set $ E_1$. Since any vertex in $D$  has at least $ s+t-2$ neighbors in $ H$, we conclude that  every vertex in $ D$ is incident to at least $ t-1$ edges in  $ E(H)\setminus E_1$.
Now, we have
\[
|E(H)|\geqslant |E(D, V(H))|\geqslant (s-1)|D|+ \frac{t-1}{2}|D|\geqslant \left(\frac{2s+t-3}{2}+o(1)\right)n.
\qedhere
\]
\end{proof}

The second lower bound in Theorem  \ref{th:bip_intro}  is stated below.

\begin{theorem}\label{lowbnd}
Let $ t\geqslant s\geqslant 2 $ and $ p\in (0, 1)$ be constants. Then, {\sl whp}
\[
\mathrm{sat}\big(\mathbbmsl{G}(n, p), K_{s,t}\big)\geqslant \left(\frac{t-s}{4p^{s-1}}+\frac{s-1}{2}+o(1)\right)n.
\]
\end{theorem}

\begin{proof}
If $ s=t$, then the assertion follows from Corollary \ref{cor}. So, assume that $ t>s$.  	
Let $ G\thicksim\mathbbmsl{G}(n,p) $ and let $ H $ be a $ K_{s,t} $-saturated subgraph of $ G$ with minimum possible number of edges. Let $ V=V(G)$. By Theorem \ref{upb},  we have $ |E(H)| = O(n) $. Consider the partition $\{A, B, C\}$ of $ V$, where
\begin{align*}
A & = \big\{x\in V \, \big| \, d_H(x) < \log n\big\},\\
B & = \left\{x \in V \, \left| \, \log n\leqslant d_H(x) \leqslant \frac{n}{\log^{s+1} n}\right.\right\},\\
C &  = \left\{x \in V \, \left| \, d_H(x) > \frac{n}{\log^{s+1} n}\right.\right\}.
\end{align*}
For any  $ y\in V$, set $ N_y=N_H(y)$ and $$ F_y=\big\{x\in V \, \big| \, |N_H(x, y)|\geqslant t-1 \big\}.$$ Moreover, let  $$ \EuScript{O}=\big\{Y\subseteq V \, \big| \, |Y|=s-1 \text{ and } |N_H(Y)|\geqslant t\big\}.$$  	Further, for any  $ Y\in \EuScript{O}$, set $ N_Y=N_H(Y)$ and $$ F_Y=\big\{x\in V \, \big| \, |N_H(\{x\}\cup Y)|= t-1 \big\}.$$
Finally,  consider the partition $ \{\EuScript{A}, \EuScript{B}, \EuScript{C} \}$ of $ \EuScript{O}$, where
\begin{align*}
\EuScript{A} &  = \{Y\in \EuScript{O} \,|\,  Y\cap A \neq \varnothing\},\\
\EuScript{B} &  = \{Y\in \EuScript{O} \,|\, Y\cap A = \varnothing  \text{ and }  Y\cap B \neq \varnothing \},\\
\EuScript{C}&  = \{Y\in \EuScript{O} \,|\,  Y\subseteq C \}.
\end{align*}
Since adding every edge $xx'\in E(G)\setminus E(H)$ to $ H$ creates a copy of  $K_{s,t} $ in $ H$, we conclude that
$ E(G)\setminus E(H) \subseteq \bigcup_{ Y\in \EuScript{O}}E_G(N_Y, F_Y) $. Therefore,  using Lemma \ref{random1}, we find that {\sl whp}
\begin{align}\label{B1}
\left|\bigcup_{ Y\in \EuScript{O}}E_G(N_Y, F_Y)\right|\geqslant |E(G)\setminus E(H)|= \frac{n^2p}{2}\big(1+ o(1)\big).
\end{align}	
Note that $ E_G(N_Y, F_Y)\subseteq E_G(N_y, F_y) $ for every $ y\in Y$, since
$ N_Y\subseteq N_y$ and $ F_Y\subseteq F_y$.
For every vertex $ y\in V$, by a double counting of the set
$ \{(x, S) \, | \,  x\in F_y,  S \subseteq N_H(x, y), \text{ and }  |S|=s \}$, we derive that
\[
|F_y|\binom{t-1}{s}\leqslant \binom{|N_y|}{s}(t-1).
\]
It follows from $ t>s$ that $ |F_y|\leqslant |N_y|^s$.
Hence,  $ |F_y|\leqslant \log^sn$ for every $ y\in A$. This gives
\begin{align}\label{B2}
\left|\bigcup_{ Y\in \EuScript{A}}E_G(N_Y, F_Y)\right| \leqslant \left|\bigcup_{ y\in {A}}E_G(N_y, F_y)\right|
\leqslant  \sum_{ y \in A} |N_y||F_y|\leqslant n (\log n) \log^sn
=n\log^{s+1} n.
\end{align}
Since  $ |E(H)|=O(n)$ {\sl whp}, we get that   $ |F_y\setminus A|=O(n/\log n)$ for each  $ y\in V$ {\sl whp}.
Using this, we may write {\sl whp}
\begin{align}\label{B3}
\left|\bigcup_{ Y\in \EuScript{B}}E_G(N_Y, F_Y)\right|&\leqslant
\left|\bigcup_{ y\in {B}}E_G(N_y, F_y)\right|\nonumber\\
&\leqslant\sum_{ y \in B} |E_G(N_y, F_y)|\nonumber\\
&\leqslant  \sum_{ y \in B} |E_G(N_y, F_y\setminus A)|+\sum_{ y \in B}|E_G(N_y, F_y\cap A)| \nonumber\\
&\leqslant\sum_{ y \in B} |N_y||F_y\setminus A| +\sum_{ y \in B}|N_y|| F_y\cap A|  \nonumber\\
&\leqslant O\left(\frac{n}{\log n}\right)\sum_{ y \in B} |N_y|+ \frac{n}{\log^{s+1}n}\sum_{x \in A}|F_x\cap B|\nonumber\\
&\leqslant O\left(\frac{n}{\log n}\right)|E(H)| +\frac{n}{\log^{s+1}n}|A|\log^{s}n\nonumber\\
&=O\left(\frac{n^2}{\log n}\right)
\end{align}
Since $|E(H)|=O(n)$ {\sl whp},   we deduce that  $ |C| = O(\log^{s+1} n)$ {\sl whp}  and so
$ |\EuScript{C}|\leqslant |C|^{s-1} = O(\log^{s^2-1} n) $ {\sl whp}.
Now, by setting $ \lambda=2$ in Corollary \ref{cor.random}, we obtain  that  {\sl whp}
\begin{align}\label{B4}
\left|\bigcup_{ Y\in \EuScript{C}}E_G(N_Y, F_Y)\right|&\leqslant
\sum_{ Y \in \EuScript{C}} |E_G(N_Y, F_Y)|\nonumber\\
&\leqslant \sum_{Y\in \EuScript{C}}\biggl(3n\log^2n+p|N_Y||F_Y|\big(1+o(1)\big)\biggr) \nonumber\\
&\leqslant 3|\EuScript{C}|n\log^{2}n+\sum_{ Y \in \EuScript{C}} p^{s}n|F_Y|\big(1+o(1)\big)\nonumber\\
&\leqslant O\left(n\log^{s^2+1}n\right)+p^{s}n\left(\sum_{Y\in \EuScript{C}}|F_Y|\right)\big(1+o(1)\big).
\end{align}
Therefore,	by \eqref{B1}--\eqref{B4}, we find that {\sl whp}
\begin{align}\label{B5}
\sum_{Y\in \EuScript{C}}|F_Y|\geqslant \frac{n}{2p^{s-1}}\big(1+o(1)\big).
\end{align}
Set $$ S=\bigcup_{{X, Y\in \EuScript{C}}\atop{X\neq Y}} N_X\cap N_Y. $$
Note that  $ |S|\leqslant \binom{|\EuScript{C}|}{2}(t-1)=O(\log^{2s^2-2}n)$ {\sl whp}. For every $Y\in \EuScript{C}$, set $ M_Y = N_Y\setminus S$.	Let
$$ F'=\left\{\left.x\in \bigcup_{Y\in \EuScript{C}}F_Y \, \right| \, |N_x\cap S|\geqslant s \right\}.$$
We claim that $ |F'|\leqslant \binom{|S|}{s}(t-1)$. To see this, suppose otherwise. By the pigeonhole principle, there is a $ t$-subset $T$ of $ F'$ such that $ |N_H(T)\cap S|\geqslant s $ which gives a copy of $ K_{s, t}$ in $ H$, a contradiction.
This proves the claim which in turn implies that $ |F'|=O(\log ^{2s^3-2s}n)$. For every $ Y\in \EuScript{C} $, set $ F'_Y= F_Y\setminus F'$.
Noting that the sets $ M_Y$ are mutually disjoint and $ F_Y\cap Y=\varnothing$ for every $Y\in \EuScript{O} $, we may write {\sl whp}
\begin{align*}
2|E(H)|&= \sum_{Y\in \EuScript{C}}\sum_{x\in M_Y}d_H(x)+\sum_{x\notin \mathop{\bigcup}_{Y\in \EuScript{C}}M_Y}d_H(x)\\
&\geqslant\sum_{Y\in \EuScript{C}} \big( |E_H(M_Y, F'_Y\setminus M_Y)|+2|E_H(M_Y, F'_Y\cap M_Y)|+|E_H(M_Y, Y)|\big)
+(s-1)\left|V\setminus \bigcup_{Y\in \EuScript{C}}M_Y\right|\\
&\geqslant\sum_{Y\in \EuScript{C}}\big((t-s)|F'_Y\setminus M_Y|+(t-s)|F'_Y\cap M_Y|+(s-1)|M_Y|\big)+ (s-1)\left(n- \sum_{Y\in \EuScript{C}}|M_Y|\right)\\
& =(s-1)n+ \sum_{Y\in \EuScript{C}}(t-s)|F'_Y|\\
& = (s-1)n+(t-s)\left(\left(\sum_{Y\in \EuScript{C}}|F_Y|\right)-|\EuScript{C}| |F'|\right)\\
&\geqslant (s-1)n+(t-s)\left(\frac{n}{2p^{s-1}}\big(1+o(1)\big)-O\left(\log^{2s^3+s^2-2s-1}n\right)\right)\\
& = \left(\frac{t-s}{2p^{s-1}}+s-1+o(1)\right)n,
\end{align*}
where the last inequality follows from  (\ref{B5}), completing the proof.
\end{proof}

We   point     out here  that Theorem \ref{th:bip_intro} is concluded from Theorems    \ref{pikh.low}  and  \ref{lowbnd}.

\begin{remark}
It is worth noting that using the proof of Theorem \ref{pikh.low}, one may improve the estimate on the number of edges of $H $ in the last paragraph of proof of Theorem \ref{lowbnd}  to obtain
\[
\mathrm{sat}(\mathbbmsl{G}(n, p), K_{s,t})\geqslant \left(\frac{t-s}{4p^{s-1}}+s-1+o(1)\right)n.
\]
For the sake of clarity of presentation we disregarded this improvement in the proof of Theorem  \ref{lowbnd}.
\end{remark}


\begin{thebibliography}{99}
	

\bibitem{Bottcher} P. Allen, J. B\"ottcher, J. Ehrenm\"uller, A. Taraz, The bandwidth theorem in sparse graphs, Adv. Comb. (2020) \#P6.

\bibitem{AF} N. Alon, Z. F\"{u}redi, Spanning subgraphs of random graphs, Graphs Combin. 8 (1992) 91--94.

\bibitem{pikh} T. Bohman, M. Fonoberova, O. Pikhurko, The saturation function of complete partite graphs, J. Comb. (2010) 149--170.

\bibitem{Cam} A. Cameron, G.J. Puleo, A lower bound on the saturation number, and graphs for which it is sharp, Discrete Math. 345 (2022) 112867.

\bibitem{DSZ} Yu.A.  Demidovich, A. Skorkin, M.     Zhukovskii, Cycle saturation in random graphs, 2022,  available    at: \url{http://arxiv.org/pdf/2109.05758.pdf}.

\bibitem{DZ} S. Demyanov, M.   Zhukovskii, Tight concentration of star saturation number in random graphs, 2022,  available    at: \url{http://arxiv.org/pdf/2212.06101.pdf}.

\bibitem{DHZ} S. Diskin, I. Hoshen, M.   Zhukovskii, A jump of the saturation number in random graphs,  2023, available    at: \url{http://arxiv.org/pdf/2303.12046.pdf}.

\bibitem{erd} P. Erd\H{o}s,  A. Hajnal,  J.W. Moon, A problem in graph theory,  Amer. Math. Monthly   71  (1964) 1107--1110.

\bibitem{fau} J.R.  Faudree,   R.J.  Faudree,  J.R.  Schmitt, A survey of minimum saturated graphs,  Electron.  J. Combin.   18  (2011) \#DS19.

\bibitem{FKM} N. Fountoulakis, R.J. Kang, C. McDiarmid, The $t$-stability number of a random graph, Electron. J. Combin. 17 (2010) \#R59.

\bibitem{Furedi} Z. F\"uredi, Y. Kim, Cycle-saturated graphs with minimum number of edges, J. Graph Theory 73 (2013)  203--215.

\bibitem{JLR}  S. Janson, T. \L uczak, A. Ruci\'{n}ski, Random Graphs,  Wiley-Interscience Series in Discrete Mathematics and Optimization,  Wiley-Interscience, New York, 2000.

\bibitem{KSZ} D. Kamaldinov, A. Skorkin, M.   Zhukovskii, Maximum sparse induced subgraphs of the binomial random graph with given number of edges,  Discrete Math. 344 (2021)  112205.

\bibitem{kas} L. K\'aszonyi,  Zs. Tuza,  Saturated graphs with minimal number of edges,   J. Graph Theory   10  (1986) 203--210.

\bibitem{kor} D. Kor\'andi,  B. Sudakov, Saturation in random graphs,  Random Structures  Algorithms    51  (2017)  169--181.

\bibitem{MT} A. Mohammadian, B. Tayfeh-Rezaie, Star saturation number of random graphs, Discrete Math. 341 (2018) 1166--1170.	

\bibitem{tur} P. Tur\'an, Eine Extremalaufgabe aus der Graphentheorie,    Mat. Fiz. Lapok  48  (1941) 436--452.

\bibitem{zyk} A.A.  Zykov, On some properties of linear complexes,   Mat. Sbornik N.S.  24 (66) (1949) 163--188.


\end{thebibliography}
\end{document}